\documentclass{IEEEtran}
\usepackage{cite}
\usepackage{amsthm}
\usepackage{amsmath,amssymb,amsfonts}
\usepackage{algorithmic}
\usepackage{graphicx,color}
\usepackage{textcomp}
\def\BibTeX{{\rm B\kern-.05em{\sc i\kern-.025em b}\kern-.08em
    T\kern-.1667em\lower.7ex\hbox{E}\kern-.125emX}}

    \usepackage{hyperref}
\hypersetup{
    colorlinks=false,
    linkcolor=black,
    filecolor=black,
    urlcolor=black,
}

\usepackage[font=small,labelfont=bf]{caption}
\usepackage{tikz}
\usetikzlibrary{shapes,matrix,decorations.shapes}
\usetikzlibrary{arrows,decorations.pathmorphing,backgrounds,positioning,fit,matrix}
\usetikzlibrary{shapes,arrows,chains}
\usetikzlibrary[calc]

\usepackage{thmtools}

\newtheorem{theorem}{Theorem}
\newtheorem{assumption}{Assumption}
\newtheorem{definition}{Definition}
\newtheorem{remark}{Remark}
\newtheorem{proposition}{Proposition}
\newtheorem{corollary}{Corollary}

\begin{document}
\title{On the Equivalence of Youla, System-level and Input-output Parameterizations}

\author{Yang Zheng, Luca Furieri, Antonis Papachristodoulou, Na Li, and Maryam Kamgarpour
\thanks{This work is supported by NSF career 1553407, AFOSR Young Investigator Program, and ONR Young Investigator Program. A. Papachristodoulou is supported by the EPSRC Grant EP/M002454/1.
L. Furieri and M. Kamgarpur are gratefully supported by ERC Starting Grant CONENE. }
\thanks{Y. Zheng and N. Li are with SEAS and CGBC, Harvard
University, Cambridge, MA 02138. (E-mails: zhengy@g.harvard.edu; nali@seas.harvard.edu). }
\thanks{L. Furieri and M. Kamgarpour are with the Automatic Control Laboratory,  ETH Zurich, Switzerland. (E-mails: \{furieril, mkamgar\}@control.ee.ethz.ch).}
%\thanks{Y. Zheng and N. Li are with John A. Paulson School of Engineering and Applied Sciences, and also  with the Harvard Center for Green Buildings and Cities, Harvard University, Cambridge, MA 02138. (E-mails: zhengy@g.harvard.edu; nali@seas.harvard.edu). }
%\thanks{L. Furieri and M. Kamgarpour are with the Automatic Control Laboratory, Department of Information Technology and Electrical Engineering, ETH Zurich, Switzerland. (E-mails: \{furieril, mkamgar\}@control.ee.ethz.ch).}
\thanks{A. Papachristodoulou is with the Department of Engineering Science , University of
Oxford, United Kingdom. (E-mail: antonis@eng.ox.ac.uk).
}
}

\maketitle

\begin{abstract}
    A convex parameterization of internally stabilizing controllers is fundamental for many controller synthesis procedures. The celebrated Youla parameterization relies on a doubly-coprime factorization of the system, while the recent system-level and input-output characterizations require no doubly-coprime factorization but a set of equality constraints for achievable closed-loop responses. In this paper, we present explicit affine mappings among Youla, system-level and input-output parameterizations. Two direct implications of the affine mappings are 1) any convex problem in Youla, system level, or input-output parameters can be equivalently and convexly formulated in any other one of these frameworks,
    % in terms of Youla, system-level, input-output parameters can be equivalently shifted from each other,
    including the convex system-level synthesis (SLS); 2)  the condition of  quadratic  invariance  (QI) is sufficient and necessary for the classical distributed control problem to  admit  an  equivalent  convex  reformulation  in  terms  of  Youla, system-level, or input-output parameters.
\end{abstract}


\section{Introduction}
\label{sec:introduction}

One of the most fundamental problems in control theory is to design a feedback controller that stabilizes a dynamical system. Additionally, one can further design an optimal controller by optimizing a certain performance measure~\cite{zhou1996robust}. It is well-known that the set of stabilizing controllers is in general non-convex, and hence, hard to optimize directly over. %over which it is hard to search directly.
Many optimal controller synthesis procedures first parameterize all stabilizing controllers and the corresponding closed-loop responses in a convex way, and then minimize relevant performance measures over the new parameter(s).

For finite dimensional linear-time-invariant (LTI) systems, the set of LTI stabilizing feedback controllers is fully characterized by the celebrated \emph{Youla parameterization}~\cite{youla1976modern}, where a doubly coprime  factorization  of  the  system is used. In~\cite{youla1976modern}, it is shown that the Youla parameterization allows for optimizing the Youla parameter (or system response) directly, instead of the controller itself, leading to a convex problem. Also, customized performance specifications on the closed-loop system can be incorporated with Youla parameterization via convex optimization~\cite{boyd1991linear}. Moreover, the foundational results of robust and optimal control are built on the Youla parameterization~\cite{francis1987course,zhou1996robust}. Note that a doubly-coprime factorization of the system must be computed as a preliminary step in Youla parameterization. Recently, a system-level parameterization~\cite{wang2019system} and an input-output parameterization~\cite{furieri2019input} were introduced to characterize the set of all LTI stabilizing controllers, with no need of computing a doubly-coprime factorization of the system \emph{a priori}. Similar to Youla, the system-level and input-output parameterizations treat certain closed-loop responses as design parameters. The controller synthesis is thus shifted from designing a controller to  designing  the  closed  loop  responses  directly.
%This treatment shifts the controller synthesis  from the design of a controller to the design of the closed loop responses.
This idea of synthesizing closed-loop responses in a convex way was extensively discussed as \emph{closed-loop convexity} in~\cite[Chapter 6]{boyd1991linear}.

The Youla~\cite{youla1976modern}, system-level~\cite{wang2019system}, input-output~\cite{furieri2019input} parameterizations are equivalent since they characterize the same set of stabilizing controllers. However, their explicit relationships have not been fully established before. The main objective of this paper is to reveal an explicit equivalence of Youla, system-level, and input-output parameterizations. In particular, we present explicit \emph{affine mappings} among the Youla parameter, system-level parameters, and input-output parameters. One direct consequence is that any convex problem in terms of Youla, system-level, input-output parameters can be equivalently and convexly formulated into any other one of these three frameworks. %shifted from each other.
Therefore, the so-called convex system-level synthesis (SLS)~\cite{wang2019system} admits an equivalent convex formulation in terms of Youla or input-output parameters. Another consequence is that if one controller synthesis task does not allow for an equivalent convex reformulation in Youla, a convex reformulation in the system-level or input-output parameterizations is not possible either. Consider the classical distributed controller synthesis task where a subspace constraint is imposed on the controller~\cite{rotkowitz2006characterization}. It has been shown that a notion of quadratic invariance (QI) is sufficient and necessary for the distributed control problem to admit an equivalent convex reformulation in the Youla parameter~\cite{rotkowitz2006characterization,lessard2015convexity}. Accordingly, the QI condition is also sufficient and necessary when using the system-level and input-output parameterizations. For systems with constraints beyond QI, we highlight that a notion of sparsity invariance (SI)~\cite{Furieri2019Sparsity} can be used to derive convex inner-approximations using Youla, system-level, or input-output characterizations.
% We note that the SLS in~\cite{wang2019system} imposes structural constraints on certain system responses directly, instead of the controller, leading to a different problem compared to the classical decentralized controller synthesis.

%System-level synthesis is another perspective, resulting a different problem.

The rest of this paper is organized as follows. We introduce some preliminaries in Section~\ref{Section:preliminaries}, and review the Youla, system-level, and input-output parameterizations in Section~\ref{Section:paramterization}. Explicit affine relationships and their implication with QI are presented in Section~\ref{Section:equivalence}. We discuss distributed controller synthesis with non-QI constraints in Section~\ref{Section:specialcase}, and conclude the paper in Section~\ref{section:conclusion}.

\noindent\emph{Notation:} We use lower and upper case letters (\emph{e.g.} $x$ and $A$) to denote vectors and matrices, respectively. Lower and upper case boldface letters (\emph{e.g.} $\mathbf{x}$ and $\mathbf{G}$) are used to denote signals and transfer matrices, respectively. For clarity, we consider discrete-time LTI systems only, but unless stated otherwise, all results can be extended to the continuous-time setting. We denote the set of real-rational proper stable transfer matrices as $\mathcal{RH}_{\infty}$. $\mathbf{G} \in \frac{1}{z}\mathcal{RH}_{\infty}$ means $\mathbf{G}$ is stable and strictly proper.

% mention the push through identity, and inverse identity

\section{Preliminaries}~\label{Section:preliminaries}
\vspace{-6mm}
\subsection{System model}
We consider discrete-time LTI systems of the form
\begin{equation} \label{eq:LTI}
        \begin{aligned}
            {x}[t+1] &= A x[t] + B_1 w[t] + B_2u[t], \\
            z[t] &= C_1 x[t] + D_{11}w[t] + D_{12}u[t], \\
            y[t] &= C_2x[t] + D_{21}w[t] + D_{22} u[t],
        \end{aligned}
    \end{equation}
    where $x[t],u[t],w[t],y[t],z[t]$ are the state vector, control action, external disturbance, measurement, and regulated output at time $t$, respectively. System~\eqref{eq:LTI} can be written as
    $$
        \mathbf{P} = \left[ \begin{array}{c|cc} A & B_1 & B_2 \\
        \hline
                                     C_1 & D_{11} & D_{12} \\
                                     C_2 & D_{21} & D_{22} \end{array} \right]
                  = \begin{bmatrix} \mathbf{P}_{11} & \mathbf{P}_{12} \\\mathbf{P}_{21} & \mathbf{P}_{22} \end{bmatrix},
    $$
    where $\mathbf{P}_{ij} = C_i(zI - A)^{-1}B_j +D_{ij}$. We refer to $\mathbf{P}$ as the open-loop plant model.

    Consider a dynamic output feedback controller
    $
        \mathbf{u} = \mathbf{K}\mathbf{y},
    $
    where $\mathbf{K}$  has a state space realization
    %has a state space realization
    \begin{equation} \label{eq:ControllerLTI}
        \begin{aligned}
            {\xi}[t+1] &= A_k \xi[t] + B_k y[t], \\
            u[t] &= C_k \xi[t] + D_{k}y[t],
        \end{aligned}
    \end{equation}
    with $\xi$ as the internal state of controller $\mathbf{K}$. We have $\mathbf{K} = C_k(zI - A_k)^{-1}B_k + D_k$. Figure~\ref{fig:LTI} shows a schematic diagram of the interconnection of  plant $\mathbf{P}$ and  controller $\mathbf{K}$.  Throughout the paper, we make the following standard assumptions.
\begin{assumption}
    Both the plant and controller realizations are stabilizable and detectable, \emph{i.e.}, $(A, B_2)$ and $(A_k, B_k)$ are stabilizable, and $(A,C_2)$ and $(A_k,C_k)$ are detectable.
\end{assumption}

\begin{assumption}
    The interconnection in Fig.~\ref{fig:LTI} is well-posed, \emph{i.e.}, $I - D_{22}D_k$ is invertible.%The plant is strictly proper, \emph{i.e.}, $D_{22} = 0.$
\end{assumption}

\subsection{Stabilization and optimal control}

%We now introduce the definition of internal stability of the system in Fig.~\ref{eq:LTI}.
\begin{definition}
    The system in Fig.~\ref{eq:LTI} is \emph{internally stable} if it is well-posed, and the states $(x[t],x_k[t])$ converge to zero as $t\rightarrow \infty$ for all initial states $(x[0],x_k[0])$ when $w[t] = 0, \forall t$.
\end{definition}

We say the controller $\mathbf{K}$ \emph{internally stabilizes} the plant $\mathbf{P}$ if the interconnected system in Fig.~\ref{eq:LTI} is {internally stable}. The set of all stabilizing controllers is defined as
\begin{equation}
    \mathcal{C}_{\text{stab}} := \{\mathbf{K} \mid \mathbf{K} \; \text{internally stabilizes} \; \mathbf{P}\}.
\end{equation}
In addition to stability, it is desirable to find a controller $\mathbf{K}$ that minimizes a suitable norm (\emph{e.g.}, $\mathcal{H}_2$ or $\mathcal{H}_{\infty}$) of the closed-loop transfer matrix from $\mathbf{w}$ to $\mathbf{z}$. This amounts to solving the following optimal control formulation~\cite{zhou1996robust}:
\begin{equation} \label{eq:OCP}
        \begin{aligned}
            \min_{\mathbf{K}} \quad &\|f(\mathbf{P},\mathbf{K})\| \\
            \text{subject to} \quad & \mathbf{K} \in \mathcal{C}_{\text{stab}},
        \end{aligned}
    \end{equation}
where $f(\mathbf{P},\mathbf{K}) = \mathbf{P}_{11} + \mathbf{P}_{12}\mathbf{K}(I - \mathbf{P}_{22}\mathbf{K})^{-1}\mathbf{P}_{21}$. It is known that set $\mathcal{C}_{\text{stab}}$ is non-convex. One can construct explicit examples where $\mathbf{K}_1, \mathbf{K}_2 \in \mathcal{C}_{\text{stab}}$ but $\frac{1}{2}(\mathbf{K}_1+ \mathbf{K}_2) \notin \mathcal{C}_{\text{stab}}$. Also, $f(\mathbf{P},\mathbf{K})$ is in general a non-convex function of $\mathbf{K}$. Therefore, problem~\eqref{eq:OCP} is non-convex in the current form.

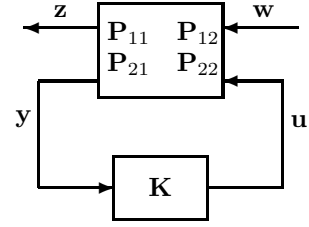
\begin{figure}[t]
\setlength{\belowcaptionskip}{0pt}
\setlength{\abovecaptionskip}{0pt}
\begin{center}\setlength{\unitlength}{0.008in}
\begin{picture}(243,135)(140,410)
\thicklines
\put(180,440){\vector(1, 0){ 50}}
\put(180,510){\line( 0, -1){ 70}}
%\put(128,440){\vector( 1, 0){ 48}}
%\put(140,445){\makebox(0,0)[lb]{$v_1$}}
%\put(181,440){\circle{10}}%\put(340,550){\line( 0, 1){ 70}}
\put(180,510){\line( 1, 0){ 40}}

%\put(340,440){\circle{10}}%\put(340,550){\line( 0, 1){ 70}}
\put(340,510){\vector( -1, 0){ 40}}
\put(340,510){\line( 0,-1){ 70.5}}
\put(290,440){\line(1, 0){ 50}}
%\put(393,440){\vector(-1, 0){ 48}}
%\put(365,445){\makebox(0,0)[lb]{$v_2$}}

%\put(140,530){\vector( 1, 0){ 80}}
%\put(300,530){\vector( 1, 0){ 80}}
%\put(230,620){\line(-1, 0){ 50}}
%\put(180,620){\line( 0,-1){ 70}}
\put(220,545){\vector( -1, 0){ 50}}
\put(350,545){\vector( -1, 0){ 50}}
%\put(340,620){\vector(-1, 0){ 50}}
\put(230,420){\framebox(60,40){}}
%\put(230,600){\framebox(60,40){}}
\put(220,500){\framebox(80,60){}}
\put(165,480){\makebox(0,0)[lb]{$\mathbf{y}$}}
\put(345,480){\makebox(0,0)[lb]{$\mathbf{u}$}}
\put(320,552){\makebox(0,0)[lb]{$\mathbf{w}$}}
\put(190,552){\makebox(0,0)[lb]{$\mathbf{z}$}}
\put(252,435){\makebox(0,0)[lb]{$\mathbf{K}$}}
\put(225,510){\makebox(0,0)[lb]{$\begin{matrix}\mathbf{P}_{11}&\mathbf{P}_{12}\\\mathbf{P}_{21}&\mathbf{P}_{22}\end{matrix}$}}
%\put(253,615){\makebox(0,0)[lb]{$\Delta$}}
%\put(158,580){\makebox(0,0)[lb]{$u_{\Delta}$}}
%\put(345,580){\makebox(0,0)[lb]{$y_{\Delta}$}}
\end{picture}\end{center}
\caption{Interconnection of the plant $\mathbf{P}$ and controller $\mathbf{K}$}
\label{fig:LTI}
\end{figure}

% \begin{figure}[t]
%   \centering
%
%   % Requires \usepackage{graphicx}
%   \includegraphics[width=0.24 \textwidth]{LTI.png}\\
%   \caption{Interconnection of the plant $\mathbf{P}$ and controller $\mathbf{K}$}\label{fig:LTI}
% \end{figure}

\section{Parameterization of Stabilizing Controllers}\label{Section:paramterization}
To solve the optimal control problem~\eqref{eq:OCP}, one common method is to derive an equivalent convex formulation via a suitable change of variables. A classical technique is the Youla parameterization~\cite{youla1976modern}. Two recent approaches are the so-called system-level parameterization (SLP)~\cite{wang2019system}, and input-output parameterization (IOP)~\cite{furieri2019input}. A common idea among these three approaches is the parameterization of all stabilizing controllers $\mathcal{C}_{\text{stab}}$ using certain closed-loop responses. We review their main results in this section.  %The main purpose of this paper is to show explicit equivalence of Youla, System-level, and Input-output characterizations.
%; for completeness,

\subsection{Youla parameterization}

The classical Youla parameterization is based on a doubly-coprime factorization of the plant $\mathbf{P}_{22}$, defined as follows.
    \begin{definition}
        A collection of stable transfer matrices, $\mathbf{U}_l, \mathbf{V}_l,\mathbf{N}_l,\mathbf{M}_l,\mathbf{U}_r, \mathbf{V}_r,\mathbf{N}_r,\mathbf{M}_r \in \mathcal{RH_{\infty}}$ is called a doubly-coprime factorization of $\mathbf{P}_{22}$ if
        $
            \mathbf{P}_{22} = \mathbf{N}_r\mathbf{M}_r^{-1} = \mathbf{M}_l^{-1}\mathbf{N}_l
        $
        and
        $$
            \begin{bmatrix} \mathbf{U}_l & -\mathbf{V}_l \\ -\mathbf{N}_l & \mathbf{M}_l\end{bmatrix}
            \begin{bmatrix} \mathbf{M}_r & \mathbf{V}_r \\ \mathbf{N}_r & \mathbf{U}_r\end{bmatrix} = I.
        $$
    \end{definition}

    Such doubly-coprime factorizations can always be computed if $\mathbf{P}_{22}$ is stabilizable and detectable~\cite{zhou1996robust}. The authors in~\cite{youla1976modern} established the following equivalence
    \begin{equation} \label{eq:youla}
       % \begin{aligned}
            \mathcal{C}_{\text{stab}} = \{\mathbf{K} = (\mathbf{V}_r - \mathbf{M}_r\mathbf{Q})(\mathbf{U}_r - \mathbf{N}_r\mathbf{Q})^{-1} \mid   \mathbf{Q} \in \mathcal{RH}_{\infty} \}\footnote{Equivalently, $\mathcal{C}_{\text{stab}}=\{(\mathbf{U}_l-\mathbf{QN}_l)^{-1}(\mathbf{V}_l-\mathbf{QM}_l)|~\mathbf{Q} \in \mathcal{RH}_\infty\}$.},
    %    \end{aligned}
    \end{equation}
    where $\mathbf{Q}$ is called the Youla parameter. Using the change of variables
    $\mathbf{K} = (\mathbf{V}_r - \mathbf{M}_r\mathbf{Q})(\mathbf{U}_r - \mathbf{N}_r\mathbf{Q})^{-1}$, it is not difficult to derive
    $$
        f(\mathbf{P},\mathbf{K}) = \mathbf{T}_{11} + \mathbf{T}_{12}\mathbf{Q}\mathbf{T}_{21},
    $$
    where $\mathbf{T}_{11} = \mathbf{P}_{11} + \mathbf{P}_{12} \mathbf{V}_{r}\mathbf{M}_{l} \mathbf{P}_{21}, \mathbf{T}_{12} = -\mathbf{P}_{12}\mathbf{M}_{r}$, and $\mathbf{T}_{21} = \mathbf{M}_{l}\mathbf{P}_{21}$. Consequently, Problem~\eqref{eq:OCP} can be equivalently reformulated in terms of the Youla parameter as
    \begin{equation} \label{eq:OCPYoula}
        \begin{aligned}
            \min_{\mathbf{Q}} \quad &\|\mathbf{T}_{11} + \mathbf{T}_{12}\mathbf{Q}\mathbf{T}_{21}\| \\
            \text{subject to} \quad & \mathbf{Q} \in \mathcal{RH}_{\infty}.
        \end{aligned}
    \end{equation}
    One direct benefit is that~\eqref{eq:OCPYoula} is convex with respect to the Youla parameter $\mathbf{Q}$.

\subsection{System-level parameterization (SLP)}
 In~\cite{wang2019system}, the authors proposed a system-level parameterization for $\mathcal{C}_{\text{stab}}$. This approach is based on the closed-loop maps from process and measurement disturbances to state and control action. In particular, assuming a strictly proper plant $\mathbf{P}_{22},$ \emph{i.e.}, $D_{22} = 0$, we use ${\delta_x}[t] = B_1 w[t]$ to denote the disturbance on the state and $\delta_y[t] = D_{21}w[t]$ to denote the disturbance on the measurement. The dynamics of plant~\eqref{eq:LTI} can be written as
    \begin{equation*} %\label{eq:LTIsls}
        \begin{aligned}
            {x}[t+1] &= A x[t] +  B_2u[t] + \delta_x[t], \\
            y[t] &= C_2x[t] + \delta_y[t].
        \end{aligned}
    \end{equation*}

    Then, with a stabilizing controller $\mathbf{u} = \mathbf{K}\mathbf{y}$, the system responses  %$\{\mathbf{R}, \mathbf{M}, \mathbf{N}, \mathbf{L}\}$
    from perturbations $(\mathbf{\delta_x}, \mathbf{\delta_y})$ to $(\mathbf{x},\mathbf{u})$ are
    \begin{equation} \label{eq:LTIsls}
        \begin{aligned}
            \begin{bmatrix} \mathbf{x} \\\mathbf{u} \end{bmatrix} = \begin{bmatrix} \mathbf{R} &  \mathbf{N}\\  \mathbf{M} &  \mathbf{L} \end{bmatrix}  \begin{bmatrix} \mathbf{\delta_x} \\ \mathbf{\delta_y} \end{bmatrix},
        \end{aligned}
    \end{equation}
    where the system responses $\{\mathbf{R}, \mathbf{M}, \mathbf{N}, \mathbf{L}\}$  are in the following affine subspace~\cite{wang2019system}
    \begin{subequations} \label{eq:slp}
        \begin{align}
            \begin{bmatrix}zI - A & -B_2 \end{bmatrix}\begin{bmatrix} \mathbf{R} &  \mathbf{N}\\  \mathbf{M} &  \mathbf{L} \end{bmatrix} & = \begin{bmatrix}I & 0\end{bmatrix},  \label{eq:slp_s1}\\
             \begin{bmatrix} \mathbf{R} &  \mathbf{N}\\  \mathbf{M} &  \mathbf{L} \end{bmatrix} \begin{bmatrix}zI - A \\ -C_2 \end{bmatrix} & = \begin{bmatrix}I \\ 0\end{bmatrix}, \label{eq:slp_s2} \\
             \mathbf{R}, \mathbf{M}, \mathbf{N} \in \frac{1}{z} \mathcal{RH}_{\infty}, \quad& \mathbf{L} \in \mathcal{RH}_{\infty}. \label{eq:slp_s3}
        \end{align}
    \end{subequations}

    In~\cite{wang2019system}, it is proved that for strictly proper\footnote{The equivalence~\eqref{eq:sls} only holds for strictly proper plants, \emph{i.e.}, $D_{22} = 0$. For a general proper plant $D_{22} \neq 0$, the authors in~\cite{wang2019system} present another controller implementation that internally stabilizes the system; see~\cite[Section III.D]{wang2019system}. Instead, Youla~\eqref{eq:youla} and input-output~\eqref{eq:iop} parameterizations work for both strictly proper and general proper plants. Throughout the paper, we assume a strictly proper plant for the system-level parameterization. } plant $\mathbf{P}_{22}$, the set of all internally stabilizing controllers can be written as
    \begin{equation} \label{eq:sls}
        \begin{aligned}
            \mathcal{C}_{\text{stab}}  = \{\mathbf{K} = \mathbf{L} - \mathbf{M}\mathbf{R}^{-1}&\mathbf{N}  \mid  \mathbf{R},   \, \mathbf{M},  \, \mathbf{N}, \, \mathbf{L} \; \text{are in the }  \\
             &\text{affine subspace~\eqref{eq:slp_s1}-\eqref{eq:slp_s3}}   \}.
        \end{aligned}
    \end{equation}
    Also, the cost function $f(\mathbf{P},\mathbf{K})$ can be expressed in terms of the system responses $\mathbf{R}, \mathbf{M}, \mathbf{N}, \mathbf{L}$. In particular, Problem~\eqref{eq:OCP} can be equivalently reformulated as %in terms of the system responses as
     \begin{equation} \label{eq:OCPsls}
        \begin{aligned}
            \min_{\mathbf{R}, \mathbf{M}, \mathbf{N}, \mathbf{L}} \quad &\left\| \begin{bmatrix} C_1 & D_{12} \end{bmatrix}\begin{bmatrix} \mathbf{R} &  \mathbf{N}\\  \mathbf{M} &  \mathbf{L} \end{bmatrix} \begin{bmatrix} B_1 \\ D_{21} \end{bmatrix}  + D_{11} \right\| \\
            \text{subject to} \quad & \eqref{eq:slp_s1}-\eqref{eq:slp_s3}.
        \end{aligned}
    \end{equation}
    It is easy to see that~\eqref{eq:OCPsls} is convex in terms of  $\mathbf{R}, \mathbf{M}, \mathbf{N}, \mathbf{L}$.

\subsection{Input-output parameterization (IOP)}

%{\color{red} picture and replace X, Y, W ,Z }

\begin{figure}[t]
  \centering
  \begin{center}\setlength{\unitlength}{0.008in}
\begin{picture}(243,115)(140,410)
\thicklines
\put(187,440){\vector(1, 0){ 43}}
\put(180,510){\vector( 0, -1){ 65}}
\put(128,440){\vector( 1, 0){ 48}}
\put(140,445){\makebox(0,0)[lb]{$\delta_{\mathbf{y}}$}}
\put(181,440){\circle{10}}%\put(340,550){\line( 0, 1){ 70}}
\put(180,510){\line( 1, 0){ 48.5}}

\put(340,510){\circle{10}}%\put(340,550){\line( 0, 1){ 70}}
\put(335,510){\vector( -1, 0){ 45}}
\put(340,439.5){\vector( 0,1){ 65.5}}
\put(290,440){\line(1, 0){ 50}}
\put(393,510){\vector(-1, 0){ 48}}
\put(365,515){\makebox(0,0)[lb]{$\delta_{\mathbf{u}}$}}

%\put(140,530){\vector( 1, 0){ 80}}
%\put(300,530){\vector( 1, 0){ 80}}
%\put(230,620){\line(-1, 0){ 50}}
%\put(180,620){\line( 0,-1){ 70}}
%\put(220,545){\vector( -1, 0){ 50}}
%\put(350,545){\vector( -1, 0){ 50}}
%\put(340,620){\vector(-1, 0){ 50}}
\put(230,420){\framebox(60,40){}}
%\put(230,600){\framebox(60,40){}}
\put(230,490){\framebox(60,40){}}
\put(200,445){\makebox(0,0)[lb]{$\mathbf{y}$}}
\put(315,515){\makebox(0,0)[lb]{$\mathbf{u}$}}
%\put(320,552){\makebox(0,0)[lb]{$w$}}
%\put(190,552){\makebox(0,0)[lb]{$z$}}
\put(252,435){\makebox(0,0)[lb]{$\mathbf{K}$}}
\put(248,505){\makebox(0,0)[lb]{$\mathbf{P}_{22}$}}
%\put(253,615){\makebox(0,0)[lb]{$\Delta$}}
%\put(158,580){\makebox(0,0)[lb]{$u_{\Delta}$}}
%\put(345,580){\makebox(0,0)[lb]{$y_{\Delta}$}}
\put(185,450){\makebox(0,0)[lb]{\footnotesize +}}
\put(167,445){\makebox(0,0)[lb]{\footnotesize +}}

\put(350,515){\makebox(0,0)[lb]{\footnotesize +}}
\put(345,498){\makebox(0,0)[lb]{\footnotesize +}}

\end{picture}\end{center}
  \caption{Input-output stability.}\label{fig:IOS}
\end{figure}
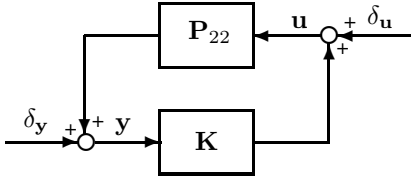

Recently, an input-output parameterization for $\mathcal{C}_{\text{stab}}$ was introduced in~\cite{furieri2019input}. As shown in Fig.~\ref{fig:IOS}, the idea is based on a classical internal stability result in terms of the following closed-loop responses
    \begin{equation} \label{eq:LTIio}
        \begin{aligned}
           % \begin{bmatrix} \mathbf{y} \\\mathbf{u} \end{bmatrix} = \begin{bmatrix} (I - \mathbf{P}_{22}\mathbf{K})^{-1} &  (I - \mathbf{P}_{22}\mathbf{K})^{-1}\mathbf{P}_{22}\\ \mathbf{K}(I - \mathbf{P}_{22}\mathbf{K})   &  (I - \mathbf{K}\mathbf{P}_{22})\end{bmatrix}  \begin{bmatrix} \mathbf{\delta_y} \\ \mathbf{\delta_u} \end{bmatrix}.
             \begin{bmatrix} \mathbf{y} \\\mathbf{u} \end{bmatrix} = \begin{bmatrix} \mathbf{Y} &  \mathbf{W}\\ \mathbf{U}   &  \mathbf{Z}\end{bmatrix}  \begin{bmatrix} \mathbf{\delta_y} \\ \mathbf{\delta_u} \end{bmatrix},
        \end{aligned}
    \end{equation}
where $\delta_{\mathbf{u}}$ is a disturbance in the input, \emph{i.e.}, $\mathbf{u} = \mathbf{K}y + \delta_{\mathbf{u}}$.
    Under Assumption 1, it is known that $\mathbf{K}$ internally stabilizes $\mathbf{P}$ if and only if the four transfer matrices in~\eqref{eq:LTIio} are stable~\cite{francis1987course}. With a stabilizing controller $\mathbf{K}$, the closed-loop responses $\mathbf{Y}, \mathbf{U}, \mathbf{W}, \mathbf{Z}$ are in the following affine subspace~\cite{furieri2019input}
    \begin{subequations} \label{Eq:Param}
	 \begin{align}
	 &\begin{bmatrix}
	 I&-\mathbf{P}_{22}
	 \end{bmatrix}\begin{bmatrix}
	 \mathbf{Y}&\mathbf{W}\\\mathbf{U}&\mathbf{Z}
	 \end{bmatrix}=\begin{bmatrix}
	 I&0
	 \end{bmatrix}\,, \label{eq:aff1}\\
	 & \begin{bmatrix}
	 \mathbf{Y}&\mathbf{W}\\\mathbf{U}&\mathbf{Z}
	 \end{bmatrix}\begin{bmatrix}
	 -\mathbf{P}_{22}\\I
	 \end{bmatrix}=\begin{bmatrix}
	 0\\I
	 \end{bmatrix}\label{eq:aff2}\,,\\
	 &\begin{matrix}
	 \mathbf{Y}, \mathbf{U}, \mathbf{W}, \mathbf{Z}\in \mathcal{RH}_\infty.
	 \end{matrix}\label{eq:aff3}
	 \end{align}
	 \end{subequations}
	
	It is shown in~\cite{furieri2019input} that the set of all internally stabilizing controllers can be represented as
    \begin{equation} \label{eq:iop}
        \begin{aligned}
            \mathcal{C}_{\text{stab}}  = \{\mathbf{K} = \mathbf{U}\mathbf{Y}^{-1}  \mid \mathbf{Y}, &\mathbf{U}, \mathbf{W},  \mathbf{Z} \; \text{are in the} \\
            &\text{affine subspace~\eqref{eq:aff1}-\eqref{eq:aff3}}   \}.
        \end{aligned}
    \end{equation}
    Furthermore, we have $f(\mathbf{P},\mathbf{K}) = \mathbf{P}_{11} + \mathbf{P}_{12}\mathbf{U}\mathbf{P}_{21}$~\cite{furieri2019input}. Accordingly, Problem~\eqref{eq:OCP} can be equivalently reformulated in terms of the system responses as
     \begin{equation} \label{eq:OCPiop}
        \begin{aligned}
            \min_{\mathbf{Y}, \mathbf{U}, \mathbf{W}, \mathbf{Z}} \quad &\left\| \mathbf{P}_{11} + \mathbf{P}_{12}\mathbf{U}\mathbf{P}_{21}\right\| \\
            \text{subject to} \quad & ~\eqref{eq:aff1}-\eqref{eq:aff3}.
        \end{aligned}
    \end{equation}
    It is easy to see that~\eqref{eq:OCPiop} is convex in terms of  $\mathbf{Y}, \mathbf{U}, \mathbf{W}, \mathbf{Z}$.

\section{Explicit Equivalence of Youla Parameterization, SLP, and IOP}
\label{Section:equivalence}

As discussed in the last section, the set of all stabilizing controllers $\mathcal{C}_{\text{stab}}$ can be parameterized in three different ways, \emph{i.e.},~\eqref{eq:youla},~\eqref{eq:sls}, and~\eqref{eq:iop}, and the optimal control problem~\eqref{eq:OCP} admits three equivalent convex reformulations,  \emph{i.e.},~\eqref{eq:OCPYoula},~\eqref{eq:OCPsls}, and~\eqref{eq:OCPiop}. Implicitly,~\eqref{eq:youla},~\eqref{eq:sls}, and~\eqref{eq:iop} are equivalent. However, an explicit relationship between Youla parameterization, SLP, and IOP is not clear from their definitions.

In this section, we present explicit \emph{affine mappings} among Youla parameterization, SLP, and IOP. The consequences are as follows: 1) any convex system-level synthesis (SLS) introduced by~\cite{wang2019system} can be equivalently reformulated into a convex problem in terms of the Youla parameter $\mathbf{Q}$ or the input-output parameters $\mathbf{Y}, \mathbf{U}, \mathbf{W}, \mathbf{Z}$, and vice versa; 2) building on the explicit affine mappings, we show that the notion of quadratic invariance~\cite{rotkowitz2006characterization} allows for equivalent convex reformulations of classical distributed optimal control in either Youla parameterization, SLP, or IOP.

\subsection{Explicit equivalence}

The explicit equivalence between Youla parameterization and IOP is presented in~\cite{furieri2019input}:
\begin{theorem}[\!\cite{furieri2019input}]
\label{th:Youla_eq}
Let $\mathbf{U}_r,\mathbf{V}_r,\mathbf{U}_l,\mathbf{V}_l,\mathbf{M}_r,\mathbf{M}_l,\mathbf{N}_r,\mathbf{N}_l$ be any  doubly-coprime factorization of $\mathbf{P}_{22}$. The following statements hold.
\begin{enumerate}
\item For any $\mathbf{Q} \in \mathcal{RH}_\infty$, the following transfer matrices
    \begin{subequations} \label{eq:youla_iop}
    \begin{align}
    &\mathbf{Y}=(\mathbf{U}_r-\mathbf{N}_r\mathbf{Q})\mathbf{M}_l\,, \label{eq:Q_to_X_1}\\
    &\mathbf{U}=(\mathbf{V}_r-\mathbf{M}_r\mathbf{Q})\mathbf{M}_l\,,\\
    &\mathbf{W}=(\mathbf{U}_r-\mathbf{N}_r\mathbf{Q})\mathbf{N}_l\,,\\
    &\mathbf{Z}=I+(\mathbf{V}_r-\mathbf{M}_r\mathbf{Q})\mathbf{N}_l\,,\label{eq:Q_to_X_2}
    \end{align}
    \end{subequations}
belong to the affine subspace \eqref{eq:aff1}-\eqref{eq:aff3} and are such that $\mathbf{U}\mathbf{Y}^{-1}=(\mathbf{V}_r-\mathbf{M}_r\mathbf{Q})(\mathbf{U}_r-\mathbf{N}_r\mathbf{Q})^{-1}$.
\item For any $(\mathbf{Y},\mathbf{U},\mathbf{W},\mathbf{Z})$ in the affine subspace  (\ref{eq:aff1})-(\ref{eq:aff3}), the transfer matrix
    \begin{equation}
    \label{eq:Youla_with_XYWZ}
    \mathbf{Q}=\mathbf{V}_l\mathbf{Y}\mathbf{U}_r-\mathbf{U}_l\mathbf{U}\mathbf{U}_r-\mathbf{V}_l\mathbf{W}\mathbf{V}_r+\mathbf{U}_l\mathbf{Z}\mathbf{V}_r-\mathbf{V}_l\mathbf{U}_r\,,
    \end{equation}
is such that $\mathbf{Q} \in \mathcal{RH}_\infty$ and $(\mathbf{V}_r-\mathbf{M}_r\mathbf{Q})(\mathbf{U}_r-\mathbf{N}_r\mathbf{Q})^{-1}=\mathbf{U}\mathbf{Y}^{-1}$.
\end{enumerate}
\end{theorem}

Theorem~\ref{th:Youla_eq} presents explicit {affine mappings} between Youla parameterization and IOP: {any element in the Youla parameterization~\eqref{eq:youla} corresponds to an element in the IOP~\eqref{eq:iop}, and they represent the same controller.} The following result presents explicit {affine mappings} between SLP and IOP.
\begin{theorem}
\label{th:slp_eq}
Consider a strictly proper plant $\mathbf{P}_{22}$, \emph{i.e.}, $D_{22} = 0$. The following statements hold.
\begin{enumerate}
\item For any $\mathbf{R}, \mathbf{M}, \mathbf{N}, \mathbf{L}$ satisfying the affine subspace~\eqref{eq:slp_s1}-\eqref{eq:slp_s3}, the transfer matrices
  \begin{subequations} \label{eq:slp-iop}
    \begin{align}
    \mathbf{Y} &= C_2\mathbf{N} + I, \label{eq:slp-iop1}\\
      \mathbf{U} &=  \mathbf{L}, \label{eq:slp-iop2}\\
      \mathbf{W} &=  C_2\mathbf{R}B_2, \label{eq:slp-iop3}\\
      \mathbf{Z} &=  \mathbf{M}B_2 + I, \label{eq:slp-iop4}
    \end{align}
    \end{subequations}
belong to the affine subspace~\eqref{eq:aff1}-\eqref{eq:aff3} and are such that
    $
        \mathbf{L} - \mathbf{M}\mathbf{R}^{-1}\mathbf{N} = \mathbf{U}\mathbf{Y}^{-1}.
    $
\item For any $\mathbf{Y}, \mathbf{U}, \mathbf{W}, \mathbf{Z}$ satisfying the affine subspace~\eqref{eq:aff1}-\eqref{eq:aff3}, the transfer matrices
  \begin{subequations} \label{eq:iop-sls}
    \begin{align}
            \mathbf{R} &= (zI - A)^{-1} + (zI - A)^{-1}B_2\mathbf{U}C_2(zI - A)^{-1} \label{eq:iop-sls1} \\
            \mathbf{M} & = \mathbf{U}C_2(zI - A)^{-1},  \label{eq:iop-sls2}\\
            \mathbf{N} &= (zI - A)^{-1}B_2\mathbf{U},  \label{eq:iop-sls3} \\
            \mathbf{L} &= \mathbf{U}, \label{eq:iop-sls4}
        \end{align}
    \end{subequations}
    belong to the affine subspace~\eqref{eq:slp_s1}-\eqref{eq:slp_s3} and are such that
    $
        \mathbf{U}\mathbf{Y}^{-1} = \mathbf{L} - \mathbf{M}\mathbf{R}^{-1}\mathbf{N}.
    $
\end{enumerate}
\end{theorem}

\begin{proof}
    \emph{Statement 1}: considering the affine relationships~\eqref{eq:slp_s1}-\eqref{eq:slp_s3} and $\mathbf{P}_{22} = C_2(zI - A)^{-1}B_2$, we have the following algebraic equalities:
    $$
        \begin{aligned}
        \mathbf{Y} - \mathbf{P}_{22}\mathbf{U} &= C_2\mathbf{N} + I - \mathbf{P}_{22}\mathbf{L} \\
        &= C_2(\mathbf{N} - (zI - A)^{-1}B_2\mathbf{L}) + I  = I, \\
        \mathbf{W} - \mathbf{P}_{22}\mathbf{Z} &=  C_2\mathbf{R}B_2 - \mathbf{P}_{22}( \mathbf{M}B_2 + I) \\
        &= C_2\left(\mathbf{R} - (zI - A)^{-1}B_2\mathbf{M} -(zI - A)^{-1}\right)B_2 \\
        &= 0, \\
        \mathbf{Y}\mathbf{P}_{22} - \mathbf{W} &= (C_2\mathbf{N} + I)\mathbf{P}_{22} - C_2\mathbf{R}B_2 \\
        &= C_2\left( \mathbf{N}C_2(zI - A)^{-1} + (zI - A)^{-1} - \mathbf{R}\right)B_2 \\
        &= 0,\\
        \end{aligned}
        $$
        $$
        \begin{aligned}
        -\mathbf{U}\mathbf{P}_{22} + \mathbf{Z} &= -\mathbf{L}\mathbf{P}_{22} + \mathbf{M}B_2 + I \\
        & = (-\mathbf{L}C_2(zI-A)^{-1} + \mathbf{M})B_2 + I \\
        &= I.
        \end{aligned}
    $$
    Therefore,~\eqref{eq:aff1} and~\eqref{eq:aff2} are satisfied. Obviously, the transfer matrices in~\eqref{eq:slp-iop} are stable, \emph{i.e.}~\eqref{eq:aff3} is satisfied.
    In addition, from~\eqref{eq:slp_s2}, we have $\mathbf{R} = (I + \mathbf{N}C_2)(zI - A)^{-1}$. Then,
    $$
        \begin{aligned}
        \mathbf{L} - \mathbf{M}\mathbf{R}^{-1}\mathbf{N} &=   \mathbf{L} - \mathbf{M}(zI - A)(I + \mathbf{N}C_2)^{-1}\mathbf{N} \\
                                                        &= \mathbf{L} - \mathbf{L}C_2(I + \mathbf{N}C_2)^{-1}\mathbf{N}  \\
                                                        &= \mathbf{L} - \mathbf{L}C_2\mathbf{N}(I + C_2\mathbf{N})^{-1} \\
                                                        & = \mathbf{L}(I + C_2\mathbf{N})^{-1} \\
                                                        & = \mathbf{U}\mathbf{Y}^{-1}.
        \end{aligned}
    $$

    \emph{Statement 2}:
    %
    % considering the affine relationships~\eqref{eq:aff1}-\eqref{eq:aff3} and $\mathbf{P}_{22} = C_2(zI - A)^{-1}B_2$, we have the following algebraic equalities:
    % $$
    %     \begin{aligned}
    %     (zI - A)\mathbf{R} - B_2\mathbf{M} &= I + B_2\mathbf{Y}C_2(zI - A)^{-1}  \\
    %     &\qquad \qquad - B_2\mathbf{Y}C_2(zI - A)^{-1} \\
    %     &= I,\\
    %     (zI - A)\mathbf{N} - B_2\mathbf{L} &= B_2\mathbf{Y} - B_2\mathbf{Y} = 0, \\
    %     \mathbf{R}(zI - A) - \mathbf{N}C_2 &= I + (zI - A)^{-1}B_2\mathbf{Y}C_2 \\
    %     &\qquad \qquad - (zI - A)^{-1}B_2\mathbf{Y}C_2  \\
    %     &= I,\\
    %      \mathbf{M}(zI - A) - \mathbf{L}C_2 &= \mathbf{Y}C_2-  \mathbf{L}C_2 = 0.
    %     \end{aligned}
    % $$
    % Thus,~\eqref{eq:slp_s1}-\eqref{eq:slp_s2} are satisfied.
    In the Appendix~\ref{Sec:stable}, we verify algebraically that the transfer matrices $\mathbf{R},\mathbf{M},\mathbf{N},\mathbf{L}$ defined in~\eqref{eq:iop-sls} are exactly the closed-loop responses in~\eqref{eq:LTIsls} with controller $\mathbf{K} = \mathbf{U}\mathbf{Y}^{-1}$.

    Since $\mathbf{K} = \mathbf{U}\mathbf{Y}^{-1}$ is internally stabilizing $\mathbf{P}_{22}$, the transfer matrices $\mathbf{R},\mathbf{M},\mathbf{N},\mathbf{L}$ defined in~\eqref{eq:iop-sls} naturally satisfy the constraints~\eqref{eq:slp_s1}-\eqref{eq:slp_s3}. % are naturally satisfied.
    %
    % we have
    % $$
    %     \begin{aligned}
    %      C_2\mathbf{N} &= C_2(sI - A)^{-1}B_2\mathbf{Y} = \mathbf{P}_{22}\mathbf{Y} = \mathbf{X} - I \in \frac{1}{z}\mathcal{RH}_{\infty}, \\
    %      \mathbf{M}B_2 &= \mathbf{Y}C_2(sI - A)^{-1}B_2 = \mathbf{Y}\mathbf{P}_{22} = \mathbf{Z} - I \in   \frac{1}{z}\mathcal{RH}_{\infty}, \\
    %      C_2\mathbf{R}B_2 &= \mathbf{P}_{22} + \mathbf{P}_{22}\mathbf{Y}\mathbf{P}_{22} = \mathbf{X}\mathbf{P}_{22} = \mathbf{W}  \in   \frac{1}{z}\mathcal{RH}_{\infty}, \\
    %       \mathbf{L} &= \mathbf{Y}  \in  \mathcal{RH}_{\infty}. \\
    %      \end{aligned}
    % $$
    %
    % Using the facts that $(A, C_2)$ is detectable and $(A,B_2)$ is stabilizable, we can show that $\mathbf{R},\mathbf{M},\mathbf{N}$ are all stable; a detailed proof is provided in the appendix.
    % %{\color{blue} Since $(A, C_2)$ is detectable and $(A,B_2)$ is stabilizable, we have $\mathbf{R},\mathbf{M},\mathbf{N} \in \frac{1}{z}\mathcal{RH}_{\infty}$. }
    Finally, we can check that the transfer matrices $\mathbf{R},\mathbf{M},\mathbf{N}, \mathbf{L}$ defined in~\eqref{eq:iop-sls} satisfy
    $
        \mathbf{U}\mathbf{Y}^{-1} = \mathbf{L} - \mathbf{M}\mathbf{R}^{-1}\mathbf{N}.
    $
    %Detailed steps are provided in the appendix.
    This completes the proof.
\end{proof}

Combining Theorems~\ref{th:Youla_eq} and~\ref{th:slp_eq}, we arrive at the explicit affine mappings between Youla parameterization and SLP, which was not provided in~\cite{wang2019system}.
\begin{theorem}
\label{th:Youla_sls}
Let $\mathbf{U}_r,\mathbf{V}_r,\mathbf{U}_l,\mathbf{V}_l,\mathbf{M}_r,\mathbf{M}_l,\mathbf{N}_r,\mathbf{N}_l$ be any  doubly-coprime factorization of the strictly proper system $\mathbf{P}_{22}$. The following statements hold.
\begin{enumerate}
\item For any $\mathbf{Q} \in \mathcal{RH}_\infty$, the following transfer matrices
  \begin{subequations} \label{eq:youla_slp}
    \begin{align}
    \mathbf{R} &= (zI - A)^{-1} + \nonumber\\
    (zI - &A)^{-1}B_2(\mathbf{V}_r-\mathbf{M}_r\mathbf{Q})\mathbf{M}_lC_2(zI - A)^{-1},  \label{eq:youla_slp1} \\
            \mathbf{M} & = (\mathbf{V}_r-\mathbf{M}_r\mathbf{Q})\mathbf{M}_lC_2(zI - A)^{-1},  \label{eq:youla_slp2}\\
            \mathbf{N} &= (zI - A)^{-1}B_2(\mathbf{V}_r-\mathbf{M}_r\mathbf{Q})\mathbf{M}_l,  \label{eq:youla_slp3}\\
            \mathbf{L} &= (\mathbf{V}_r-\mathbf{M}_r\mathbf{Q})\mathbf{M}_l, \label{eq:youla_slp4}
    \end{align}
    \end{subequations}
belong to the affine subspace \eqref{eq:slp_s1}-\eqref{eq:slp_s3} and are such that $\mathbf{L} - \mathbf{M}\mathbf{R}^{-1}\mathbf{N}=(\mathbf{V}_r-\mathbf{M}_r\mathbf{Q})(\mathbf{U}_r-\mathbf{N}_r\mathbf{Q})^{-1}$.
\item For any $(\mathbf{R},\mathbf{M},\mathbf{N},\mathbf{L})$ in the affine subspace  \eqref{eq:slp_s1}-\eqref{eq:slp_s3}, the transfer matrix
    \begin{equation} \label{eq:Youla_with_RMNL}
        \begin{aligned}
        \mathbf{Q} =&\mathbf{V}_lC_2\mathbf{N}\mathbf{U}_r-\mathbf{U}_l\mathbf{L}\mathbf{U}_r-\mathbf{V}_l C_2\mathbf{R}B_2\mathbf{V}_r \\
       & \qquad \qquad \qquad \qquad +\mathbf{U}_l\mathbf{M}B_2\mathbf{V}_r + \mathbf{U}_l\mathbf{V}_r
    \end{aligned}
    \end{equation}
is such that $\mathbf{Q} \in \mathcal{RH}_\infty$ and $(\mathbf{V}_r-\mathbf{M}_r\mathbf{Q})(\mathbf{U}_r-\mathbf{N}_r\mathbf{Q})^{-1}=\mathbf{L} - \mathbf{M}\mathbf{R}^{-1}\mathbf{N}$.
\end{enumerate}
\end{theorem}
\begin{proof}
    Statement 1 directly follows by  combining the statement 1 of Theorem~\ref{th:Youla_eq} with the statement 2 of Theorem~\ref{th:slp_eq}.

    Combining the statement 2 of Theorem~\ref{th:Youla_eq} with the statement 1 of Theorem~\ref{th:slp_eq} leads to
    \begin{equation*} %\label{eq:Youla_with_RMNL}
        \begin{aligned}
        \mathbf{Q} &=\mathbf{V}_l(C_2\mathbf{N} + I)\mathbf{U}_r-\mathbf{U}_l\mathbf{L}\mathbf{U}_r-\mathbf{V}_l C_2\mathbf{R}B_2\mathbf{V}_r \\
        & \qquad +\mathbf{U}_l(\mathbf{M}B_2 + I)\mathbf{V}_r-\mathbf{V}_l\mathbf{U}_r\,, \\
        &=\mathbf{V}_lC_2\mathbf{N}\mathbf{U}_r-\mathbf{U}_l\mathbf{L}\mathbf{U}_r-\mathbf{V}_l C_2\mathbf{R}B_2\mathbf{V}_r \\
        & \qquad +\mathbf{U}_l\mathbf{M}B_2\mathbf{V}_r + \mathbf{U}_l\mathbf{V}_r.
    \end{aligned}
    \end{equation*}
    This completes the proof.
\end{proof}

An overview of the equivalence of Youla parameterization, SLP, and IOP is shown in Fig.~\ref{Fig:Equivalence}.

\begin{remark}[Closed-loop convexity]

In both SLP and IOP, the parameters $(\mathbf{R},\mathbf{M},\mathbf{N},\mathbf{L})$ and $(\mathbf{Y},\mathbf{U},\mathbf{W},\mathbf{Z})$ have explicit and distinct physical interpretations as corresponding closed-loop transfer matrices. Also, the Youla parameter $\mathbf{Q}$ can be viewed as a closed-loop transfer matrix when the plant $\mathbf{P}_{22}$ is stable (see Remark~\ref{remark:stableQ}). In this sense, Youla parameterization, SLP, and IOP all shift the controller synthesis task from the design of a controller in~\eqref{eq:OCP}, which is non-convex, to the design of closed loop responses, resulting in convex formulations~\eqref{eq:OCPYoula},~\eqref{eq:OCPsls}, and~\eqref{eq:OCPiop}. Note that this idea of closed-loop convexity has been extensively discussed in the book~\cite{boyd1991linear}, and a comprehensive historical note is given in~\cite[Chapter 16.3]{boyd1991linear}.

%Closed-loop convexity, Styphen Boyd's book. Infinite-dimensional systems, ritz approximation, FIR approximation

\end{remark}

\begin{remark}[Numerical computation]
   After computing a doubly-coprime factorization of the plant, the Youla parameter $\mathbf{Q}$ is free in $\mathcal{RH}_{\infty}$ for parameterizing $\mathcal{C}_{\emph{\text{stab}}}$, and there are no equality constraints for achievable closed-loop responses. This feature allows to reformulate Problem~\eqref{eq:OCP} as a model matching problem~\eqref{eq:OCPYoula}, which can be reduced to the Nehari problem and then solved via the state-space method in~\cite{francis1987course}. Instead, both SLP and IOP do not require to compute a doubly-coprime factorization, but have explicit affine constraints for achievable closed-loop responses. % for parameterizing $\mathcal{C}_{\text{stab}}$.  %. % (Theorem~\ref{th:Youla_eq} shows that the explicit affine constraints in IOP can be eliminated by~\eqref{eq:youla_iop}).
    %They
    %and give two kernel space representations of the achievable system responses.
    %Due to the explicit
    Since the decision variables in constraints \eqref{eq:slp_s1}-\eqref{eq:slp_s3} and~\eqref{eq:aff1}-\eqref{eq:aff3} are infinite dimensional, there is no immediately efficient numerical method for solving~\eqref{eq:OCPsls} or~\eqref{eq:OCPiop}. The Ritz approximation~\cite[Chapter 15]{boyd1991linear} is one method for solving infinite dimensional optimization problems. Specifically, for discrete-time systems, the finite impulse response (FIR) approximation is a practical choice~\cite{wang2019system,furieri2019input}.

   % affine relationship, shown in Fig.~3. Image representation (no explicit equality constraints, good for model matching design), and kernel representation (explicit equality constraints)

   % Douly coprime factorization allows for eliminating the explicit equality constraints. Model matching problem

\end{remark}

% \begin{remark}[Controller implementation]
%   Based on the results, their controller implementation can also be realized with the Youla framework.  {\color{red} Do we need to mention controller implementation? For Youla or IOP, we usually need to explicitly comput the controller, while SLA allows an implementation based on $\mathbf{R},\mathbf{M},\mathbf{N},\mathbf{L}$ that is internally stabilizing. In my opinion, this is one major benefit of SLA. }
% \end{remark}

\begin{figure}[t]
    \centering
    \setlength{\abovecaptionskip}{6 PT}
    \setlength{\belowcaptionskip}{0em}
	\footnotesize
	\begin{tikzpicture}[every node/.style={minimum width=1cm, minimum height=1cm,text centered,align=center}]
        \node[draw,circle] (A) at (0,0) {Youla};
        \node[draw,circle] (B) at (-2.3,-3.0) {SLP};
        \node[draw,circle] (C) at (2.3,-3.0) {IOP};
        \node[draw,rectangle] (D) at (0,-1.8) {$\mathcal{C}_{\text{stab}}$};

        \draw[latex'-latex'] (A) -- node[sloped, anchor=center, above = -3mm,font=\footnotesize] {\eqref{eq:youla}} (D);
	    \draw[latex'-latex'] (B) -- node[sloped, anchor=center, above =-3mm,,font=\footnotesize] {\eqref{eq:sls}} (D);
        \draw[latex'-latex'] (C) -- node[sloped, anchor=center, above = -3mm,font=\footnotesize] {\eqref{eq:iop}} (D);

        %\draw[latex'-latex'] (A) -- node[rotate=0,above=0em,font=\footnotesize] {affine} (B);
        \draw[latex'-latex'] (A) -- node[sloped, anchor=center, above = -2mm,font=\footnotesize] {Theorem~\ref{th:Youla_sls}} (B);

        \draw[latex'-latex'] (A) --
        node[sloped, anchor=center, above = -2mm, font=\footnotesize] {Theorem~\ref{th:Youla_eq}} (C);

        \draw[latex'-latex'] (B) -- node[sloped, anchor=center, above = -2mm,font=\footnotesize] {Theorem~\ref{th:slp_eq}} (C);
	\end{tikzpicture}
	\caption{Equivalence of Youla paramterization, System-level parameterization (SLP), and Input-output parameterization (IOP).}
    \label{Fig:Equivalence}
    \vspace{-4mm}
\end{figure}
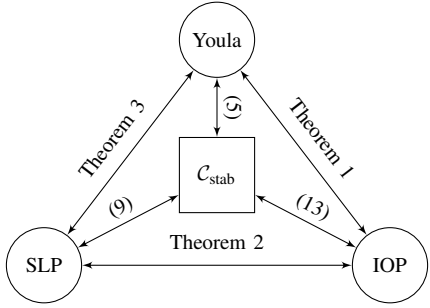

\vspace{-3mm}
\subsection{Convex system-level synthesis}

In~\cite{wang2019system}, the authors introduced a general framework of system-level synthesis (SLS), which defines ``the broadest
known class of constrained optimal control problems that can be solved using convex programming'' (cf.~\cite{anderson2019system}). % the largest known class of convex structured optimal control problems.
Thanks to the full equivalence in Theorems~\ref{th:Youla_eq}-\ref{th:Youla_sls}, we can show that 1) any SLS problem can be equivalently formulated in the Youla or input-output framework, 2) any convex SLS can be addressed by solving a convex problem in terms of Youla parameter $\mathbf{Q}$ or input-output parameters $\mathbf{Y}, \mathbf{U}, \mathbf{W}, \mathbf{Z}$.

    Let $g(\cdot)$ be a functional capturing a desired measure of the performance of the plant $\mathbf{P}_{22}$, and let $\mathcal{S}$ be a system-level constraint. The SLS problem in~\cite{wang2019system} is posed as
    \begin{equation} \label{eq:slsproblem}
        \begin{aligned}
            \min_{\mathbf{R}, \mathbf{M}, \mathbf{N}, \mathbf{L}} \quad & g(\mathbf{R}, \mathbf{M}, \mathbf{N}, \mathbf{L}) \\
            \text{subject to} \quad & \eqref{eq:slp_s1}-\eqref{eq:slp_s3}, \\
            & \begin{bmatrix} \mathbf{R} & \mathbf{N} \\
             \mathbf{M} & \mathbf{L} \end{bmatrix} \in \mathcal{S}.
        \end{aligned}
    \end{equation}
We refer the interested reader to~\cite{wang2019system} for a detailed discussion of SLS. Then, we have the following result.

\begin{theorem} \label{theo:equivalence}
   Let $\mathbf{U}_r,\mathbf{V}_r,\mathbf{U}_l,\mathbf{V}_l,\mathbf{M}_r,\mathbf{M}_l,\mathbf{N}_r,\mathbf{N}_l$ be any  doubly-coprime factorization of the strictly proper system $\mathbf{P}_{22}$. The following statements hold.
    \begin{enumerate}
        \item The SLS problem~\eqref{eq:slsproblem} is equivalent to the following problem in Youla parameter $\mathbf{Q}$,
         \begin{equation} \label{eq:sls-youla}
        \begin{aligned}
            \min_{\mathbf{Q}} \quad\; & g_1(\mathbf{Q}) \\
            \text{subject to} \quad  & \begin{bmatrix} f_1(\mathbf{Q}) & f_3(\mathbf{Q}) \\
             f_2(\mathbf{Q}) & f_4(\mathbf{Q}) \end{bmatrix} \in \mathcal{S},
        \end{aligned}
    \end{equation}
    where $f_1(\mathbf{Q}), f_2(\mathbf{Q}), f_3(\mathbf{Q}), f_4(\mathbf{Q})$ are defined by~\eqref{eq:youla_slp1} -\eqref{eq:youla_slp4}, respectively, and $$
        g_1(\mathbf{Q}) := g\left(f_1(\mathbf{Q}), f_2(\mathbf{Q}), f_3(\mathbf{Q}), f_4(\mathbf{Q})\right).
    $$
        \item The SLS problem~\eqref{eq:slsproblem} is  equivalent to the following problem in input-output parameters $\mathbf{Y}, \mathbf{U}, \mathbf{W}, \mathbf{Z}$,
         \begin{equation} \label{eq:sls-iop}
        \begin{aligned}
            \min_{\mathbf{Y},\mathbf{U},\mathbf{W},\mathbf{Z}} \quad\; & \hat{g}_1(\mathbf{U}) \\
            \text{subject to} \quad  &  ~\eqref{eq:aff1}-\eqref{eq:aff3}\\
            \quad &\begin{bmatrix} \hat{f}_1(\mathbf{U}) & \hat{f}_3(\mathbf{U}) \\
             \hat{f}_2(\mathbf{U}) & \hat{f}_4(\mathbf{U}) \end{bmatrix} \in \mathcal{S},
        \end{aligned}
    \end{equation}
    where $\hat{f}_1(\mathbf{U}), \hat{f}_2(\mathbf{U}), \hat{f}_3(\mathbf{U}), \hat{f}_4(\mathbf{U})$ are defined by~\eqref{eq:iop-sls1} -\eqref{eq:iop-sls4}, respectively, and $$
        \hat{g}_1(\mathbf{U}) := g\left(\hat{f}_1(\mathbf{U}), \hat{f}_2(\mathbf{U}), \hat{f}_3(\mathbf{U}), \hat{f}_4(\mathbf{U})\right).
    $$
        \item If the SLS problem~\eqref{eq:slsproblem} is convex, then Problems~\eqref{eq:sls-youla} and~\eqref{eq:sls-iop} are both convex.
    \end{enumerate}
   \end{theorem}

\begin{proof}
   The first two statements directly follow from Theorems~\ref{th:slp_eq} and~\ref{th:Youla_sls}.
   The last statement follows from the facts that $f_i(\mathbf{Q})$ and $\hat{f}_i(\mathbf{Q}), i = 1, \ldots, 4$, are all affine. Then, if $\mathcal{S}$ is a convex set and $g(\cdot)$ is a convex functional, the constraint in~\eqref{eq:sls-youla} (resp. ~\eqref{eq:sls-iop}) defines a convex set in $\mathbf{Q}$ (resp. $\mathbf{Y}, \mathbf{U}, \mathbf{W}, \mathbf{Z}$), and $g_1(\cdot)$ (or $\hat{g}_1(\cdot)$) is convex. %Therefore, both Problem~\eqref{eq:sls-youla} and Problem~\eqref{eq:sls-iop} are convex.
\end{proof}

\subsection{Distributed optimal control and quadratic invariance (QI)}

Unlike SLS, which impose constraints on closed-loop responses (see~\eqref{eq:slsproblem}), the classical distributed optimal control problem typically considers a subspace constraint $\mathcal{L}$ on the controller $\mathbf{K}$, which is formulated as~\cite{rotkowitz2006characterization, sabuau2014youla, qi2004structured}
\begin{equation} \label{eq:OCPsparsity}
        \begin{aligned}
            \min_{\mathbf{K}} \quad &\|f(\mathbf{P},\mathbf{K})\| \\
            \text{subject to} \quad & \mathbf{K} \in \mathcal{C}_{\text{stab}} \cap \mathcal{L}.
        \end{aligned}
    \end{equation}
It is shown in~\cite{rotkowitz2006characterization, sabuau2014youla} that if the subspace constraint $\mathcal{L}$ is \emph{quadratically invariant} (QI) under $\mathbf{P}_{22}$ (\emph{i.e.}, $\mathbf{K}\mathbf{P}_{22}\mathbf{K} \in \mathcal{L}, \forall \mathbf{K} \in \mathcal{L}$), then we have
$$
    \begin{aligned}
    \mathcal{C}_{\text{stab}} \cap \mathcal{L} =  \{\mathbf{K} = &(\mathbf{V}_r - \mathbf{M}_r\mathbf{Q})(\mathbf{U}_r - \mathbf{N}_r\mathbf{Q})^{-1} \mid \\
    & (\mathbf{V}_r-\mathbf{M}_r\mathbf{Q})\mathbf{M}_l \in \mathcal{L},   \mathbf{Q} \in \mathcal{RH}_{\infty}\}.
    \end{aligned}
$$
Problem~\eqref{eq:OCPsparsity} can thus be equivalently formulated as a convex problem in $\mathbf{Q}$~\cite{rotkowitz2006characterization, sabuau2014youla},
 \begin{equation} \label{eq:OCPYoula_sparsity}
        \begin{aligned}
            \min_{\mathbf{Q}} \quad &\|\mathbf{T}_{11} + \mathbf{T}_{12}\mathbf{Q}\mathbf{T}_{21}\| \\
            \text{subject to} \quad &  (\mathbf{V}_r-\mathbf{M}_r\mathbf{Q})\mathbf{M}_l \in \mathcal{L}, \\
            & \mathbf{Q} \in \mathcal{RH}_{\infty}.
        \end{aligned}
    \end{equation}

Considering the equivalence shown in Theorems~\ref{th:Youla_eq} and~\ref{th:Youla_sls}, the following corollaries are immediate.
\begin{corollary}[QI with IOP]
\label{prop:iop_structured}
    If $\mathcal{L}$ is QI under $\mathbf{P}_{22}$, then
\begin{enumerate}
    \item We have $$
    \begin{aligned}
    \mathcal{C}_{\text{stab}} \cap \mathcal{L} = \{\mathbf{K} = &\mathbf{U}\mathbf{Y}^{-1}  \mid \mathbf{Y}, \mathbf{U}, \mathbf{W},  \mathbf{Z} \; \text{are in the} \\
            &\quad\text{affine subspace~\eqref{eq:aff1}-\eqref{eq:aff3}}, \mathbf{U} \in \mathcal{L}   \}.
    \end{aligned}
$$
    \item Problem~\eqref{eq:OCPsparsity} can be equivalently formulated as a convex problem
      \begin{equation} \label{eq:OCPiop_s}
        \begin{aligned}
            \min_{\mathbf{Y}, \mathbf{U}, \mathbf{W}, \mathbf{Z}} \quad &\left\| \mathbf{P}_{11} + \mathbf{P}_{12}\mathbf{U}\mathbf{P}_{21}\right\| \\
            \text{subject to} \quad &\eqref{eq:aff1}-\eqref{eq:aff3}, \\
            & \; \mathbf{U} \in \mathcal{L}.
        \end{aligned}
    \end{equation}
\end{enumerate}

\end{corollary}

\begin{corollary}[QI with SLA] \label{prop:sls_structured}
    If $\mathcal{L}$ is QI under $\mathbf{P}_{22}$, then
\begin{enumerate}
    \item We have $$
    \begin{aligned}
    \mathcal{C}_{\text{stab}} \cap \mathcal{L} = \{\mathbf{K} &= \mathbf{L} - \mathbf{M}\mathbf{R}^{-1}\mathbf{N}  \mid  \mathbf{R},   \, \mathbf{M},  \, \mathbf{N}, \, \mathbf{L} \; \text{are }  \\
             &\text{in the affine subspace~\eqref{eq:slp_s1}-\eqref{eq:slp_s3}}, \mathbf{L} \in \mathcal{L}   \}.
    \end{aligned}
$$
    \item Problem~\eqref{eq:OCPsparsity} can be equivalently formulated as a convex problem
\end{enumerate}
      \begin{equation} \label{eq:OCPiop_s}
       \begin{aligned}
            \min_{\mathbf{R}, \mathbf{M}, \mathbf{N}, \mathbf{L}} \quad &\left\| \begin{bmatrix} C_1 & D_{12} \end{bmatrix}\begin{bmatrix} \mathbf{R} &  \mathbf{N}\\  \mathbf{M} &  \mathbf{L} \end{bmatrix} \begin{bmatrix} B_1 \\ D_{21} \end{bmatrix}  + D_{11}\right\| \\
            \text{subject to} \quad & \eqref{eq:slp_s1}-\eqref{eq:slp_s3}, \\
            & \mathbf{L} \in \mathcal{L}.
        \end{aligned}
    \end{equation}

\end{corollary}

%\begin{remark}
    Corollary~\ref{prop:iop_structured} is the same as Theorem~3 of \cite{furieri2019input} and Corollary~\ref{prop:sls_structured} is consistent with Theorem 3 of~\cite{wang2019system}.
    One main insight is that the specialized  proofs in  \cite{furieri2019input,wang2019system} may be not needed anymore, thanks to the explicit affine mappings between Youla, SLP and IOP.
    We also note that the original proof of Theorem 3 in~\cite{wang2019system} is not complete: it relies on that the affine mapping $\mathbf{L} = (\mathbf{V}_r-\mathbf{M}_r\mathbf{Q})\mathbf{M}_l$ is invertible. However, given $\mathbf{L} \in \mathcal{RH}_{\infty}$, it is not immediate to see that $\mathbf{Q} = \mathbf{M}_r^{-1}(\mathbf{V}_r-\mathbf{L}\mathbf{M}_r^{-1})$ is stable. We complete this fact via the construction of $\mathbf{Q}$ in~\eqref{eq:Youla_with_RMNL}.
%\end{remark}

\begin{remark}
    It should be noted that SLS~\eqref{eq:slsproblem} and the classical distributed control problem~\eqref{eq:OCPsparsity} are two distinct formulations: 1) the former imposes constraints on closed-loop responses while the latter imposes a constraint on controller $\mathbf{K}$; 2) feasibility of the former does not imply feasibility of the latter, and vice-versa. Only when the QI property holds, can Problem~\eqref{eq:OCPsparsity} be equivalently reformulated into a convex problem in terms of Youla, system-level, or input-output parameters.  Based on the results in~\cite{lessard2015convexity}, QI is necessary for the existence of such equivalent convex reformulation. For systems with QI constraints, SLS~\eqref{eq:slsproblem} can be equivalent to the classical problem~\eqref{eq:OCPsparsity}, as shown in Corollary~\ref{prop:sls_structured}; for the cases beyond QI, they are not directly comparable.
\end{remark}

%different formulation compared to SLS

%\section{Special cases}
\section{Distributed optimal control with non-QI constraints}
\label{Section:specialcase}

In this section, we highlight that for systems with non-QI constraints, we may derive convex approximations of~\eqref{eq:OCPsparsity} using Youla, system-level, or input-output parameters. In certain cases,  a  globally  optimal  solution  can  still  be  obtained. Our approximation procedure is consistent with the idea of sparsity invariance (SI)~\cite{Furieri2019Sparsity}.
In particular, we consider Example 1 in~\cite{wang2019system, anderson2019system}. We first present simplified versions of Youla, system-level, and input-output parameterizations for special cases of state feedback (for completeness, other simplified versions for stable plants are presented in Appendix~\ref{section:special}). Then, we show that Example 1 can be solved exactly using  Youla, system-level, or input-output parameters via convex optimization.

%In this section, we show that Youla parameterization, SLP, and IOP can be simplified for some special cases. Also, the example 1 in~\cite{wang2019system, anderson2019system} can be solved using Youla parameter or the input-output parameters and convex optimization.

\subsection{Simplified parameterizations for state feedback}

In~\cite{wang2019system}, it is shown that for state feedback where $C_2 = I, D_{22} = 0$, the set of internally stabilizing controllers is
 \begin{equation} \label{eq:slp-stable-state}
 \begin{aligned}
            \mathcal{C}_{\text{stab}} = \{\mathbf{K} = \mathbf{M}\mathbf{R}^{-1} \bigm|  & \begin{bmatrix} (zI -A) & -B_2  \end{bmatrix} \begin{bmatrix} \mathbf{R} \\ \mathbf{M} \end{bmatrix} = I,  \\
            &\qquad \qquad \mathbf{M},\mathbf{R} \in \frac{1}{z}\mathcal{RH}_{\infty}
           \}.
           \end{aligned}
       \end{equation}
    The proof in~\cite{wang2019system} is directly based on the definition of internal stability. As expected, this special case~\eqref{eq:slp-stable-state} can be reduced from the general case~\eqref{eq:sls} from purely algebraic operations. We provide this alternative proof in Appendix~\ref{section:proofB}.

    For IOP and Youla parameterization, simplifications are possible with further assumptions.
    \begin{corollary}[Input-output parameterization] \label{coro:iopstate}
         Suppose $C_2 = I, D_{22}= 0$ and $B_2$ is invertible. We have
         \begin{equation} \label{eq:iop-state}
         \begin{aligned}
            \mathcal{C}_{\text{stab}} = \bigg\{\mathbf{K} = (\mathbf{Z}-I)&\mathbf{W}^{-1} \bigm|
              \begin{bmatrix} I & -\mathbf{P}_{22} \end{bmatrix} \begin{bmatrix} \mathbf{W} \\ \mathbf{Z} \end{bmatrix} = 0 \\ &  \mathbf{Z} \in \mathcal{RH}_{\infty},  \mathbf{W} \in \frac{1}{z}\mathcal{RH}_{\infty}\bigg\}.
             \end{aligned}
       \end{equation}
    \end{corollary}

\begin{proof}
   %The necessary part is easy. We now prove the sufficiency part.
  %
  We show that any controller in~\eqref{eq:iop-state} is an internally stabilizing controller in~\eqref{eq:iop}. The other direction is similar. Given any $\mathbf{W}, \mathbf{Z}$ satisfying the constraints in~\eqref{eq:iop-state}, we define
   $
%   \begin{aligned}
         \mathbf{U} = (\mathbf{Z}-I)B_2^{-1}(zI - A) \in \mathcal{RH}_{\infty},
        \mathbf{Y} =  \mathbf{W}B_2^{-1}(zI - A) \in \mathcal{RH}_{\infty}.
 %  \end{aligned}
     $
 %
%   $$
%   \begin{aligned}
%          \mathbf{U} &= (\mathbf{Z}-I)B_2^{-1}(zI - A) \in \mathcal{RH}_{\infty}, \\
%         \mathbf{Y} &=  \mathbf{W}B_2^{-1}(zI - A) \in \mathcal{RH}_{\infty}.
%   \end{aligned}
%      $$
     Then, we can easily verify
     $$
     \begin{aligned}
        \mathbf{Z}-\mathbf{U}\mathbf{P}_{22} = I, \mathbf{W}-\mathbf{Y}\mathbf{P}_{22} = 0,
        \mathbf{Y}-\mathbf{P}_{22}\mathbf{U} = I.
     \end{aligned}
     $$
     Thus, $\mathbf{Y},\mathbf{U},\mathbf{W},\mathbf{Z} $ above satisfy~\eqref{eq:aff1}-\eqref{eq:aff3}.
     We also have
     $$
        \begin{aligned}
        \mathbf{U}\mathbf{Y}^{-1} &=  (\mathbf{Z}-I)B_2^{-1}(zI - A)(\mathbf{W}B_2^{-1}(zI - A))^{-1} \\
        &= (\mathbf{Z}-I)\mathbf{W}^{-1}.
        \end{aligned}
     $$
    %  $
    %   %  \begin{aligned}
    %     \mathbf{U}\mathbf{Y}^{-1} =  (\mathbf{Z}-I)B_2^{-1}(zI - A)(\mathbf{W}B_2^{-1}(zI - A))^{-1} \\
    %     = (\mathbf{Z}-I)\mathbf{W}^{-1}.
    %   % \end{aligned}
    %  $
     This completes the proof.
\end{proof}

  \begin{corollary}[Youla parameterization]
   \label{Corollary:YoulaState}
       Suppose $C_2 = B_2 = I, D_{22} = 0$. We have\footnote{Note that Corollary~\ref{Corollary:YoulaState} is only valid in discrete-time systems, since the doubly-coprime factorization~\eqref{eq:YoulaCo_state} has no counterpart in continuous time.}
    $$
    \begin{aligned}
     \mathcal{C}_{\text{stab}}  = \bigg\{\mathbf{K} = \left(-A - (I- \frac{1}{z}A)\mathbf{Q}\right)&\left(I - \frac{1}{z}\mathbf{Q}\right)^{-1},  \\
     &\qquad \mathbf{Q} \in \mathcal{RH}_{\infty} \bigg\}
     \end{aligned}
      $$
   \end{corollary}
    \begin{proof}
      The proof directly follows by choosing the following doubly-coprime factorization:
    \begin{equation} \label{eq:YoulaCo_state}
        \begin{aligned}
       \mathbf{U}_l &= I, \mathbf{V}_l = -A, \mathbf{N}_l = \frac{1}{z}I, \mathbf{M}_l = I - \frac{1}{z}A,\\
         \mathbf{U}_r &= I, \mathbf{V}_r = -A, \mathbf{N}_r = \frac{1}{z}I, \mathbf{M}_r = I - \frac{1}{z}A.\\
         \end{aligned}
    \end{equation}
    \end{proof}

% \begin{remark}
%     It is not surprising that SLP seems to have a better ability to parameterize stabilizing controllers in state feedback. This is because SLP explicitly keeps the state $x$ and directly considers the state disturbance $\mathbf{\delta}_x$, leading to closed-loop responses that are directly related to state feedback. The equivalence of Youla, system-level, input-output parameterizations does not mean they can the same number of parameters or constraints in special cases.
% \end{remark}

\subsection{Example 1 in~\cite{wang2019system, anderson2019system}}
% Here, we show that Example 1 in~\cite{wang2019system, anderson2019system} can be solved using convex programming via Youla parameterization or IOP, in addition to SLP.
% %
Consider the following optimal control problem, which is Example 1 in~\cite{wang2019system, anderson2019system},
\begin{equation} \label{eq:Example1}
\begin{aligned}
\min_{\mathbf{K}} \quad &\lim_{T \rightarrow \infty}\frac{1}{T}\sum_{t=0}^T \mathbb{E}||x[t]||_2^2\\
	\text{subject to}\quad &~x[t+1]=Ax[t]+u[t]+w[t],\\
	\quad &~\mathbf{u}=\mathbf{K}\mathbf{x}, % K(z) \in \text{Sparse}(A^{\text{bin}})\,,
\end{aligned}
\end{equation}
 where disturbance $w[t] \stackrel{\text{i.i.d}}{\sim} \mathcal{N}(0,I)$. It can be verified (\emph{e.g.}, via solving the discrete-time algebraic Riccati equation) that the globally optimal solution is the static feedback given by $ \mathbf{K}= -A$.
Assume that $A$ is sparse and let its supports define the adjacency matrix of a graph $\mathcal{G}$. Then, the optimal controller has a particular structure according to $\mathcal{G}$.
%In this case, the optimal controller can be implemented in a localized manner according to $\mathcal{G}$.

Now suppose that we attempt to solve  problem~\eqref{eq:Example1} by converting it to
its equivalent $\mathcal{H}_2$ optimal control problem in the form of~\eqref{eq:OCPsparsity}, where the constraint $\mathcal{L}$ corresponds to the sparsity pattern of $A$ (see the Example 1 in~\cite{wang2019system, anderson2019system} for a precise definition). Since~\eqref{eq:OCPsparsity} is not convex in its present form, a certain reformulation is required for numerical computation, \emph{e.g.}, using Youla parameterization, SLP, or IOP.

%Built on the results in~\cite{lessard2015convexity}, we have the following proposition.
\begin{proposition}
 If the graph $\mathcal{G}$ is strongly connected, then Problem~\eqref{eq:Example1} with a sparsity constraint $\mathbf{K} \in \mathcal{L}$ in the form of~\eqref{eq:OCPsparsity} does not admit any equivalent convex reformulation in Youla, or SLP, or IOP.
\end{proposition}
\begin{proof}
    If $\mathcal{G}$ is strongly connected, then the sparsity constraint $\mathcal{L}$ is not QI under $\mathbf{P}_{22} = (zI - A)^{-1}$, since $\mathbf{P}_{22}$ is a dense transfer matrix and it fails to satisfy $\mathbf{K}\mathbf{P}_{22}\mathbf{K} \in \mathcal{L}, \forall \mathbf{K} \in\mathcal{L}$. According to~\cite{lessard2015convexity}, QI is necessary for the existence of an \emph{equivalent} convex reformulation in Youla parameter $\mathbf{Q}$ for~\eqref{eq:OCPsparsity}. The equivalence in Theorem~\ref{theo:equivalence} prevents any \emph{equivalent} convex reformulation via SLP or IOP as well.
\end{proof}

Although there is no equivalent convex reformulation when $\mathcal{G}$ is strongly connected, we could still develop a certain \emph{convex approximation} of~\eqref{eq:OCPsparsity} in Youla parameterization, SLP, or IOP. In the following, we use $\mathcal{I}$ to denote a diagonal structure.
%Convex approximation returns a globally optimal solution

\begin{enumerate}
  \item \emph{SLP:} %see Section IV.C of~\cite{wang2019system}.
  As suggested by~\cite{wang2019system}, we can add the  constraints
  $
    \mathbf{M} \in \mathcal{L},   \mathbf{R} \in \mathcal{I}
  $
  to Problem~\eqref{eq:OCPsls}, leading to a convex approximation of~\eqref{eq:OCPsparsity}. It can be checked that $\mathbf{R}= \frac{1}{z}I$ and $\mathbf{M}= -\frac{1}{z}A$ is the optimal solution, recovering the globally optimal controller $\mathbf{K} = \mathbf{M}\mathbf{R}^{-1} = -A$.

  \item \emph{Youla parameterization:} We use the simplified Youla parameterization in Corollary~\ref{Corollary:YoulaState}, and add the following constraints
  $
    -A - (I- \frac{1}{z}A)\textbf{Q} \in \mathcal{L}, I - \frac{1}{z}\textbf{Q} \in \mathcal{I},
  $
  to Problem~\eqref{eq:OCPYoula}. This leads to a convex program. We can check that the optimal solution is $\textbf{Q} = 0$, leading to

  $$
    \textbf{K} = \left(-A - (I- \frac{1}{z}A)\textbf{Q}\right)\left(I - \frac{1}{z}\textbf{Q}\right)^{-1} =  -A.
  $$

  \item \emph{IOP:} Since $C_2 = I, B_2 = I$ is invertible, we can use the result in Corollary~\ref{coro:iopstate}. Then, we introduce constraints
  $
    \mathbf{Z} - I \in \mathcal{L}, \mathbf{W} \in \mathcal{I}
  $ to Problem~\eqref{eq:OCPiop}, leading to a convex program. We can check that the solution $\textbf{W} = \frac{1}{z}I$ and $\textbf{Z} = I - \frac{1}{z}A$ is optimal.  Then,
 $
    \textbf{K} = (\textbf{Z}- I)\textbf{W}^{-1} = -A.
 $
\end{enumerate}

%cite sparsity invariance

% This is expected since the structured optimal solution in Example 1 is known analytically, and Youla, system-level and input-output parameterizations

\begin{remark}[Sparsity invariance and beyond QI]
    In the procedures above, we choose separate subspace constraints for the factors of $\mathbf{K}$ in the following form
    \begin{equation} \label{eq:SI}
           \mathbf{S} \in \mathcal{L}, \mathbf{T} \in \mathcal{I} \quad \Rightarrow  \quad \mathbf{K} = \mathbf{S}\mathbf{T}^{-1} \in \mathcal{L},
    \end{equation}
    where $\mathbf{S}, \mathbf{T}$ denote appropriate transfer matrices in Youla, system-level, and input-output parameterizations. Obviously, this choice leads to a convex inner-approximation of~\eqref{eq:OCPsparsity} since the feasible region of $\mathbf{K}$ is narrowed. For this simple instance, the globally optimal solution is parameterized when using~\eqref{eq:SI}. Thus, the globally optimal controller can be found using Youla, system-level or input-output parameters via convex optimization. However, as observed in~\cite{Furieri2019Sparsity}, this procedure has no guarantee of optimility for general constraints beyond QI using either of Youa parameterization, SLP or IOP.
    %which is not addressed by neither Youla parameterization, SLP, nor IOP.

   We note that the property~\eqref{eq:SI} is a special case of sparsity invariance (SI)~\cite{Furieri2019Sparsity}. There may exist other subspace choices for $\mathbf{S}, \mathbf{T}$ satisfying $\mathbf{S}\mathbf{T}^{-1} \in \mathcal{L}$, which still return a structured controller $\mathbf{K} \in \mathcal{L}$. Indeed, the notion of SI goes beyond QI for sparsity constraints, as it includes QI as a special case. We refer the interested reader to~\cite{Furieri2019Sparsity} for details.
\end{remark}

% \begin{remark}[Beyond QI constraints]
%     {\color{blue}Note that all the procedures above are all convex approximations of~\eqref{eq:OCPsparsity}, but they incidentally return the globally optimal solution. The property of QI is not needed for this simple instance since we already know its globally optimal solution analytically. But QI is required to serve as a certificate of the global optimality for solving general~\eqref{eq:OCPsparsity}. In other words, for subspace constraints beyond QI, we can derive convex restrictions/relaxations for~\eqref{eq:OCPsparsity} which may still return a globally optimal solution, but there is no guarantee of optimility.  }
% \end{remark}

\section{Conclusion} \label{section:conclusion}
In this paper, we have presented an explicit equivalence of Youla, system-level, and input-output parameterizations for the set of internally stabilizing controllers. A doubly-coprime factorization of the system can be considered as a way to eliminate the explicit equality constraints in SLP and IOP. Indeed, both SLP and IOP have four parameters; but due to the equality constraints, SLP and IOP have the same degree of freedom as Youla parameterization.
%This equivalence clarifies that Youla parameterization requires a doubly-coprime factorization and has no equality constraints, while SLP and IOP require no doubly-coprime factorization but have explicit equality constraints.
% We have also clarified how to incorporate QI in Youla parameterization, SLP, or IOP to  convexify the classical distributed control problem.

%It is clear that Youla parameterization/SLP and IOP

We remark that the equivalence of Youla, SLP, and IOP does not indicate they offer the same computational features. One parameterization may be better suited for a particular context. For instance, it seems that SLP is more convenient for the case of state feedback, which has found applications in quantifying sample complexity of LQR problems~\cite{dean2017sample}; IOP seems to better suit  for the case of output feedback as it exclusively deals with the maps from inputs to outputs without explicitly touching the system state; and Youla parameterization is more convenient when a doubly-coprime factorization is available \emph{a priori}.  It is interesting to investigate whether there exist other parameterizations of stabilizing controllers that suit for a particular control application. Finally, we note that Youla, SLP, and IOP naturally suit for parameterizing dynamical controllers, but none of them can parameterize the set of static stabilizing controllers in a convex way. Thus, QI is not relevant for structured static controller synthesis, and this problem deserves further investigations.
%
%Lyapunov notion is more relevant. QI is not relevant for structured static controller design
\vspace{-2mm}

\appendix

% \subsection{An explicit example}

% Show the set of stability controller is not convex

% the set of static controllers is not convex

\subsection{Proof of Statement 2 in Theorem~\ref{th:slp_eq}} \label{Sec:stable}

Given any $\mathbf{Y}, \mathbf{U}, \mathbf{W}, \mathbf{Z}$ satisfying the affine subspace~\eqref{eq:aff1}-\eqref{eq:aff3}, we know that $\mathbf{K} = \mathbf{U}\mathbf{Y}^{-1}$ internally stabilizes the plant $\mathbf{P}_{22}$. In the following, we verify that the transfer matrices $\mathbf{R},\mathbf{M},\mathbf{N},\mathbf{L}$ defined in~\eqref{eq:iop-sls} are exactly the closed-loop responses in~\eqref{eq:LTIsls} with controller $\mathbf{K} = \mathbf{U}\mathbf{Y}^{-1}$.

Recall that $\mathbf{P}_{22}$ is strictly proper, \emph{i.e.}, $\mathbf{P}_{22} = C_2(zI - A)^{-1}B_2$. Then, we can verify the following equation:
$$
  \begin{aligned}
                \mathbf{R} &= (zI - A)^{-1} + (zI - A)^{-1}B_2\mathbf{U}C_2(zI - A)^{-1}  \\
              %  &=(zI - A)^{-1}(I + B_2\mathbf{U}C_2(zI % - A)^{-1}) \\
              &=\left[(I + B_2\mathbf{U}C_2(zI - A)^{-1})^{-1}(zI - A)\right]^{-1}
                \\
                &= \left[ zI - A - (I + B_2\mathbf{U}C_2(zI - A)^{-1})^{-1}B_2\mathbf{U}C_2\right]^{-1}\\
                &= \left[ zI - A - B_2\mathbf{U}(I + C_2(zI - A)^{-1}B_2\mathbf{U})^{-1}C_2\right]^{-1} \\
                &= \left(zI - A - B_2\mathbf{U}\mathbf{Y}^{-1}C_2\right)^{-1}\\
                &= (zI - A - B_2\mathbf{K}C_2)^{-1}
                \end{aligned}
$$
Also, we can verify
$$
    \begin{aligned}
        \mathbf{M} & = \mathbf{U}C_2(zI - A)^{-1}=(I - \mathbf{K}\mathbf{P}_{22})^{-1}\mathbf{K}C_2(zI - A)^{-1} \\
        &=\mathbf{K}C_2(zI - A)^{-1}(I - B_2\mathbf{K}C_2(zI - A)^{-1})^{-1} \\
        &=\mathbf{K}C_2(zI-A - B_2\mathbf{K}C_2)^{-1} \\
        &=\mathbf{K}C_2\mathbf{R}
    \end{aligned}
$$
Similarly, we have
$        \mathbf{N} = (zI - A)^{-1}B_2\mathbf{U}  = \mathbf{R}B_2\mathbf{K}
            \nonumber,
            \mathbf{L} = \mathbf{U} = \mathbf{K}(I - \mathbf{P}_{22}\mathbf{K})^{-1}. $
Then, the transfer matrices $\mathbf{R},\mathbf{M},\mathbf{N},\mathbf{L}$ are exactly the closed-loop responses in~\eqref{eq:LTIsls} with $\mathbf{K} = \mathbf{U}\mathbf{Y}^{-1}$.

%Consider a state space realization of $\mathbf{K} = \mathbf{Y}\mathbf{X}^{-1}$ as~\eqref{eq:ControllerLTI}. Recall that the plant $\mathbf{P}_{22}$ is strictly proper, \emph{i.e.}, $D_{22} = 0$.

% Now, routine calculations give the transfer matrix $\mathbf{X}, \mathbf{Y}, \mathbf{W}, \mathbf{Z}$ in terms of the state space realization:
% {
% %\small
%   \begin{equation*} %\label{eq:LTIio}
%         \begin{aligned}
%           \begin{bmatrix} \mathbf{X} &  \mathbf{W}\\ \mathbf{Y}   &  \mathbf{Z}\end{bmatrix}  = \hat{D}^{-1}\left(\hat{C}(zI - \hat{A})^{-1}\hat{B} + \hat{D}\right)\hat{D}^{-1},
%         \end{aligned}
%     \end{equation*}
% }
% where
% $$
% \hat{B} = \begin{bmatrix}  & B_k\\B_2&  \end{bmatrix}, \hat{C} = \begin{bmatrix} C_2 & \\&C_k  \end{bmatrix},  \hat{D} = \begin{bmatrix} I & 0 \\-D_k& I \end{bmatrix},
% $$
% and
% $$
%     \hat{A} = \begin{bmatrix} A & \\& A_k \end{bmatrix} + \begin{bmatrix}  & B_k \\B_2&  \end{bmatrix}\begin{bmatrix} I & 0 \\-D_k& I \end{bmatrix}^{-1}\begin{bmatrix} C_2 & \\&C_k  \end{bmatrix}.
% $$

\subsection{Proof of~\eqref{eq:slp-stable-state}} \label{section:proofB}
  We show that any controller in~\eqref{eq:slp-stable-state} is an internally stabilizing controller in~\eqref{eq:sls}. The other direction is similar.
Consider any $\mathbf{R}, \mathbf{M}\in \frac{1}{z}\mathcal{RH}_{\infty}$ satisfying
$$  \begin{bmatrix} (zI -A) & -B_2  \end{bmatrix} \begin{bmatrix} \mathbf{R} \\ \mathbf{M} \end{bmatrix} = I.
$$
Upon defining $        \mathbf{L} = \mathbf{M}(zI - A),
        \mathbf{N}=\mathbf{R}(zI - A) - I,
$
    % \begin{equation} \label{eq:StateNL}
    %     \mathbf{L} = \mathbf{M}(zI - A),
    %     \mathbf{N}=\mathbf{R}(zI - A) - I,\\
    % \end{equation}
it is easy to see $ \mathbf{N},  \mathbf{L} \in \mathcal{RH}_{\infty}$. Also, one can straightforwardly verify that $\mathbf{N},  \mathbf{L}$ and $\mathbf{R},  \mathbf{M}$ above satisfy~\eqref{eq:slp_s1}-\eqref{eq:slp_s2} when $C_2 = I$. It is routinely to verify that
$
      %  \begin{aligned}
            \textbf{L} - \textbf{M}\textbf{R}^{-1}\textbf{N}  =  \textbf{M}(sI - A) -  \textbf{M}\textbf{R}^{-1}(\textbf{R}(sI - A) - I)
             = \textbf{M}\textbf{R}^{-1}.
       % \end{aligned}
$
It remains to check that $\mathbf{N}$ defined above is strictly proper. This fact follows from~\eqref{eq:slp_s2}  that
$
    \mathbf{N} = (zI - A)^{-1}B_2 \mathbf{L},
$
indicating that~\eqref{eq:slp_s3} also hold. Thus, the general parameterization~\eqref{eq:sls} can be reduced to~\eqref{eq:slp-stable-state}. % from algebraic operations.

\subsection{Stable plants} \label{section:special}
When $\mathbf{P}_{22} \in \mathcal{RH}_{\infty}$, we show that Youla, SLP, and IOP can be simplified, and only two paramters are required in SLP/IOP. %are all reduced into the same form.

    \begin{proposition} \label{prop:stable}
        If $\mathbf{P}_{22} \in \mathcal{RH}_{\infty}$, we have:
        \begin{enumerate}
          \item Youla parameterization can be reduced to
          \begin{equation} \label{eq:youla-stable}
            \mathcal{C}_{\text{stab}} = \{\mathbf{K} = - \mathbf{Q}(I - \mathbf{P}_{22}\mathbf{Q})^{-1} \mid   \mathbf{Q} \in \mathcal{RH}_{\infty} \}.
       \end{equation}

          \item For strictly proper $\mathbf{P}_{22}$, SLP can be reduced to
 \begin{equation} \label{eq:slp-stable}
 \begin{aligned}
            \mathcal{C}_{\text{stab}} = \bigg\{\mathbf{K} = &\mathbf{L}(C_2\mathbf{N} + I)^{-1} \bigm|  \mathbf{L} \in \mathcal{RH}_{\infty}, \\
            &\begin{bmatrix} (zI -A) & -B_2  \end{bmatrix} \begin{bmatrix} \mathbf{N} \\ \mathbf{L} \end{bmatrix} = 0
          \bigg \}.
\end{aligned}
       \end{equation}

          \item IOP can be reduced to
          \begin{equation} \label{eq:iop-stable}
          \begin{aligned}
            \mathcal{C}_{\text{stab}} = \bigg\{\mathbf{K} = \mathbf{U}\mathbf{Y}^{-1} \bigm|  \begin{bmatrix} I & -\mathbf{P}_{22} \end{bmatrix} \begin{bmatrix} \mathbf{Y} \\ \mathbf{U} \end{bmatrix} = I,  \\
            \mathbf{U} \in \mathcal{RH}_{\infty}
           \bigg\}.
           \end{aligned}
       \end{equation}

        \end{enumerate}
    \end{proposition}

    \emph{Proof:} The proof is directly from the following observations.
    \begin{enumerate}
      \item If $\mathbf{P}_{22} \in \mathcal{RH}_{\infty}$, a doubly-coprime factorization of $\mathbf{P}_{22}$ can be trivially chosen as $
        \mathbf{U}_l = I, \mathbf{V}_l = 0, \mathbf{N}_l = \mathbf{P}_{22}, \mathbf{M}_l = I,
         \mathbf{U}_r = I, \mathbf{V}_r = 0, \mathbf{N}_r = \mathbf{P}_{22}, \mathbf{M}_r = I.
        $
  %
    %   \begin{equation*} %\label{eq:dc-stable}
    %   \begin{aligned}
    %     \mathbf{U}_l &= I, \mathbf{V}_l = 0, \mathbf{N}_l = \mathbf{P}_{22}, \mathbf{M}_l = I,\\
    %      \mathbf{U}_r &= I, \mathbf{V}_r = 0, \mathbf{N}_r = \mathbf{P}_{22}, \mathbf{M}_r = I.\\
    %     \end{aligned}
    %   \end{equation*}
      Then, the parameterization~\eqref{eq:youla} is reduced to~\eqref{eq:youla-stable}.
      \item Given $\mathbf{N}, \mathbf{L}$ in~\eqref{eq:slp-stable}, we define
    %   $$
    %     \begin{aligned}
    %       \mathbf{N} &= (zI -A)^{-1}B_2 \mathbf{L},  \\
    %       \mathbf{R} &= (zI -A)^{-1} + \mathbf{N}C_2(zI - A)^{-1}, \\
    %       \mathbf{M} &= \mathbf{L}C_2(zI - A)^{-1}.  \\
    %     \end{aligned}
    %   $$
      $
        %\begin{aligned}
           \mathbf{R} = (zI -A)^{-1} + \mathbf{N}C_2(zI - A)^{-1},
           \mathbf{M} = \mathbf{L}C_2(zI - A)^{-1}.  %\\
       % \end{aligned}
      $
      Considering $(zI - A)^{-1} \in \mathcal{RH}_{\infty}$, if $\mathbf{L} \in \mathcal{RH}_{\infty}$,  we have $ \mathbf{N}, \mathbf{R}, \mathbf{M} \in \frac{1}{z}\mathcal{RH}_{\infty}$. It can be verified that the  $\mathbf{R}, \mathbf{M}, \mathbf{N}, \mathbf{L}$ above satisfies~\eqref{eq:slp_s1}-\eqref{eq:slp_s3} when~\eqref{eq:slp-stable} holds. Also, we have
      $$
        \begin{aligned}
            \mathbf{L} - \mathbf{M}\mathbf{R}^{-1}\mathbf{N} %&=  \mathbf{L} - \mathbf{L}C_2(zI - A)^{-1}\left((zI -A)^{-1} + \mathbf{N}C_2(zI - A)^{-1}\right)^{-1}\mathbf{N} \\
            & = \mathbf{L} -\mathbf{L}C_2(I + \mathbf{N}C_2)^{-1}\mathbf{N} \\
            & = \mathbf{L}(C_2\mathbf{N} + I)^{-1}.
        \end{aligned}
      $$
    %   $
    %   % \begin{aligned}
    %         \mathbf{L} - \mathbf{M}\mathbf{R}^{-1}\mathbf{N}
    %          = \mathbf{L} -\mathbf{L}C_2(I + \mathbf{N}C_2)^{-1}\mathbf{N} %\\
    %          = \mathbf{L}(C_2\mathbf{N} + I)^{-1}.
    %   %  \end{aligned}
    %   $
      Thus,~\eqref{eq:slp_s1}-\eqref{eq:slp_s3} can be reduced to~\eqref{eq:slp-stable}.

      \item Upon defining
          $
      %  \begin{aligned}
            \mathbf{Y} = I + \mathbf{P}_{22}\mathbf{U},
            \mathbf{Z} = I + \mathbf{U}\mathbf{P}_{22},
            \mathbf{W} = \mathbf{P}_{22} \mathbf{Z},
     %   \end{aligned}
    $
    % $$
    %     \begin{aligned}
    %         \mathbf{Y} &= I + \mathbf{P}_{22}\mathbf{U},\\
    %         \mathbf{Z} &= I + \mathbf{U}\mathbf{P}_{22}, \\
    %         \mathbf{W} &= \mathbf{P}_{22} \mathbf{Z},
    %     \end{aligned}
    % $$
    we have $\mathbf{Y}, \mathbf{W}, \mathbf{Z} \in \mathcal{RH}_{\infty}$
    if $\mathbf{P}_{22}, \mathbf{U} \in \mathcal{RH}_{\infty}$. Also, the  $\mathbf{Y}, \mathbf{U}, \mathbf{W}, \mathbf{Z}$ above satisfies~\eqref{eq:aff1}-\eqref{eq:aff2} if~\eqref{eq:iop-stable} holds. Thus,~\eqref{eq:aff1}-\eqref{eq:aff3} can be reduced to~\eqref{eq:iop-stable}.
    \end{enumerate}

    \begin{remark}\label{remark:stableQ}
        The first statement in Proposition~\ref{prop:stable} is a classical result~\cite[Corollary 5.5]{zhou1996robust}. We note that parameterizations~\eqref{eq:slp-stable} and~\eqref{eq:iop-stable} are identical to~\eqref{eq:youla-stable} by noticing that
        $
            \mathbf{L} = \mathbf{U} = - \mathbf{Q}.
        $ They are all reduced to the same form.
       % This fact can also be seen from~\eqref{eq:Youla_with_XYWZ} and~\eqref{eq:Youla_with_RMNL} by subsituting~\eqref{eq:dc-stable}.
    \end{remark}

\bibliographystyle{IEEEtran}
\bibliography{IEEEabrv,references}

\end{document}